\documentclass[a4paper]{article}

\usepackage[hidelinks]{hyperref}
\usepackage{geometry}

\usepackage{amsmath}
\usepackage{amsfonts}
\usepackage{amssymb}
\usepackage{bm}
\usepackage{bbm}
\usepackage{amsthm}
\usepackage{mathtools}

\theoremstyle{plain}
  \newtheorem{theorem}{Theorem}
  \newtheorem{lemma}{Lemma}

\newcommand{\bsz}{{\boldsymbol{z}}}
\newcommand{\bsh}{{\boldsymbol{h}}}
\newcommand{\bsx}{{\boldsymbol{x}}}
\newcommand{\bsDelta}{{\boldsymbol{\Delta}}}
\newcommand{\bsgamma}{{\boldsymbol{\gamma}}}
\newcommand{\bstau}{{\boldsymbol{\tau}}}

\newcommand{\bszero}{{\boldsymbol{0}}}
\newcommand{\bbmE}{{\mathbbm{E}}}
\newcommand{\bbmP}{{\mathbbm{P}}}
\newcommand{\bbmone}{\mathbbm{1}}
\newcommand{\calP}{\mathcal{P}}

\newcommand{\calH}{\mathcal{H}}

\newcommand{\fraku}{\mathfrak{u}}
\newcommand{\frakv}{\mathfrak{v}}
\newcommand{\rme}{\mathrm{e}}
\newcommand{\rmi}{\mathrm{i}}

\newcommand{\N}{\mathbb{N}}
\newcommand{\Z}{\mathbb{Z}}

\newcommand{\rd}{\,\mathrm{d}}
\newcommand{\ceil}[1]{\left\lceil #1 \right\rceil}

\DeclareMathOperator{\supp}{supp}
\newcommand{\Pn}{\calP_n}
\DeclareRobustCommand{\stirling}{\genfrac\{\}{0pt}{}}
\newcommand{\labstirling}[2]{a(#1,#2)}
\newcommand{\st}{\: : \:}

\makeatletter
\newcommand{\vast}{\bBigg@{4}}
\newcommand{\Vast}{\bBigg@{5}}
\newcommand{\norm}[1]{\|#1\|_{d,\alpha,\bsgamma}}
\newcommand{\ralpha}{r_{\alpha,\bsgamma}}
\newcommand{\omegafixed}{\omega_{n,\bsz}}
\newcommand{\eran}[1]{e_{#1,\alpha,\bsgamma}^{\text{\textup{ran}}}}
\newcommand{\erms}[1]{e_{#1,\alpha,\bsgamma}^{\text{\textup{rms}}}}

\newcommand{\spaceH}{\calH_{d,\alpha,\bsgamma}}

\newcommand{\muquant}{\mu_{d,\alpha,\bsgamma}(\lambda)}
\newcommand{\Zdnull}{\Z^d \setminus \{\bszero\}}

\newcommand{\Gset}[1]{G_\lambda^{(#1)}}
\newcommand{\Gnset}{G_{n,\lambda}}

\newcommand{\Qranpshift}{Q_{d,n,\bsz}^{\text{\textup{ran-pr-sh}}}}

\newcommand{\Qranp}{Q_{d,n,\bsz}^{\text{\textup{ran-pr}}}}

\begin{document}
\allowdisplaybreaks

\title{A randomised lattice rule algorithm \\
with pre-determined generating vector \\
and random number of points \\
for Korobov spaces with $0 < \alpha \le 1/2$}
\author{Dirk Nuyens and Laurence Wilkes \\[2mm] Department of Computer Science, KU Leuven}
\date{December 2023}
\maketitle

\begin{abstract}
In previous work \cite{KNW2023}, we showed that a lattice rule with a pre-determined generating vector but random number of points can achieve the near optimal convergence of $O(n^{-\alpha-1/2+\epsilon})$, $\epsilon > 0$, for the worst case expected error, commonly referred to as the randomised error, for numerical integration of high-dimensional functions in the Korobov space with smoothness $\alpha > 1/2$.
Compared to the optimal deterministic rate of $O(n^{-\alpha+\epsilon})$, $\epsilon > 0$, such a randomised algorithm is capable of an extra half in the rate of convergence.
In this paper, we show that a pre-determined generating vector also exists in the case of $0 < \alpha \le 1/2$.
Also here we obtain the near optimal convergence of $O(n^{-\alpha-1/2+\epsilon})$, $\epsilon > 0$; or in more detail, we obtain $O(\sqrt{r} \, n^{-\alpha-1/2+1/(2r)+\epsilon'})$ which holds for any choices of $\epsilon' > 0$ and $r \in \N$ with $r > 1/(2\alpha)$.
\end{abstract}

\section{Introduction}\label{section:Introduction}

We study the numerical approximation of the integral,
\begin{align}\label{eq:Id}
I_d(f)
&:=
\int_{[0,1]^d} f(\bsx) \rd \bsx
,
\end{align}
for multivariate functions which belong to a certain \emph{weighted Sobolev space}, with the norm expressed in terms of Fourier coefficients, called the \emph{Korobov space}, denoted by $\spaceH$. Specifically, we are interested in functions which have a low smoothness parameter $\alpha$, being $0 < \alpha \le 1/2$, where $\alpha$ controls the $\ell_2$ summability of the Fourier coefficients, see~\eqref{eq:H} later, but all statements in this paper hold for all $\alpha > 0$.
However, when $\alpha \le 1/2$ the function is not necessarily continuous on the torus $[0,1)^d$ anymore, therefore the Fourier series does not necessarily converge pointwise and we will employ the expected value over a \emph{random shift}.

To approximate~\eqref{eq:Id} we will use a \emph{shifted rank-$1$ lattice rule}, see, e.g., \cite{N1992,SJ1994,DKP2022}, with $n$ sample points, a generating vector $\bsz \in \Z^d$ and a shift $\bsDelta \in [0,1)^d$,
\begin{equation}\label{eq:LR}
    Q_{d,n,\bsz,\bsDelta}(f)
    :=
    \frac{1}{n} \sum_{k=0}^{n-1} f\!\left(\left\{\frac{\bsz k}{n} + \bsDelta\right\} \bmod 1\right)
    .
\end{equation}
A standard \emph{rank-$1$ lattice rule} is one which uses no shift,
\begin{equation}\label{eq:SLR}
    Q_{d,n,\bsz}(f)
    :=
    \frac{1}{n} \sum_{k=0}^{n-1} f\!\left(\frac{\bsz k \bmod n}{n}\right)
    ,
\end{equation}
or, equivalently, uses the zero shift,
\[
Q_{d,n,\bsz}(f) = Q_{d,n,\bsz,\bszero}(f)
.
\]
For fixed $n$, clearly the generating vector $\bsz \in \Z^d$ could be chosen really poorly, e.g., take $\bsz = (1, 1, \ldots, 1)$ and the zero shift, which will result in all the sample points being on the diagonal of the unit cube.
Therefore we are interested in finding a good generating vector for a given $n$.

In the \emph{deterministic case} (where necessarily $\alpha > 1/2$), when given a number of sample points $n$, one would seek to determine a generating vector $\bsz \in \Z^d$, which will minimise the \emph{worst case error}, which is the supremum of the error of the algorithm over all functions in the unit ball of the function space \cite{SJ1994,DKP2022}. In this case, the worst case error for a shifted lattice rule in the Korobov space with $\alpha > 1/2$ is independent of the choice of shift, and so~\eqref{eq:LR} and~\eqref{eq:SLR} both attain the same worst case error.

In the \emph{randomised setting}, \cite{KKNU2019,DGS2022,NW2008}, we allow algorithms to use randomness; e.g., for a lattice rule this could be in the way of selecting any or all of $n$, $\bsz$ and $\bsDelta$ randomly, so as to minimise the \emph{worst case expected error}, also known as the \emph{randomised error}, formally defined later in~\eqref{eq:eran}.
Algorithms for numerical integration can use randomness in multiple ways, e.g., one can consider randomly shifted lattice rules \cite{SKJ2002}; or randomly selected number of points and generating vectors \cite{KKNU2019,DGS2022,B1961}; or randomly selected number of points with a pre-determined generating vector \cite{KNW2023}, or take the median of rules with randomly selected generating vectors \cite{GlE2022}; or randomly scaled, and possibly randomly shifted, Frolov rules \cite{KN2017,U2017,F1976}; or Owen-scrambled digital nets \cite{O1995,O1997,D2007}; or just plain Monte Carlo sampling.
Not all of these randomised algorithms improve the order of convergence, and only very few of them obtain the optimal randomised error for the Korobov space that we consider in this paper, which is $O(n^{-\alpha-1/2})$, see \cite{U2017}.

The randomised algorithm we use in this paper is an extension of the one in \cite{KNW2023}, which, for given $n$, selects a random number of samples from a set of prime numbers within a certain range, for $n \ge 2$,
\[
\Pn
:=
\big\{n/2 < p \leq n : p \text{ prime}\big\}
,
\]
and then uses a pre-determined generating vector to create a rank-1 lattice rule for integrating a function. Despite its simplicity, this algorithm has been proven to have the near optimal worst case expected error $O(n^{-\alpha-1/2+\epsilon})$, $\epsilon > 0$, for functions in $\spaceH$, with $\alpha > 1/2$, for a carefully chosen $\bsz \in \Z^d$, see \cite{KNW2023}.
This algorithm is given as
\begin{align}\label{eq:Qran}
\Qranp(f,\omega)
&:=
Q_{d,p(\omega),\bsz}(f)
\qquad \text{with } p \sim \mathcal{U}(\Pn)
,
\end{align}
where the notation $p \sim \mathcal{U}(\Pn)$ means that the function $\omega \mapsto p(\omega)$ is a random variable distributed with the uniform distribution on $\Pn$.

When $0 < \alpha \le 1/2$ and we want to use a lattice rule we are required to employ a random shift.
The algorithm that we analyse in this paper is given by
\begin{align}\label{eq:Qran-shift}
\Qranpshift(f, \omega) := Q_{d,p(\omega_1),\bsz,\bsDelta(\omega_2)}(f)
\quad \text{with } p \sim \mathcal{U}(\Pn), \; \bsDelta \sim \mathcal{U}([0,1)^d)
,
\end{align}
where the notation means that $p(\omega_1)$ and $\bsDelta(\omega_2)$, as functions of the random event $\omega = (\omega_1,\omega_2)$, are independent random variables with the uniform distribution on $\Pn$ and $[0,1)^d$ respectively.
Compared to~\eqref{eq:Qran}, the only extra element is the necessity of a random shift which is taken i.i.d.\ uniformly on $[0,1)$ for every dimension.
The proof that this new algorithm, $\Qranpshift$, reaches the near optimal randomised error $O(n^{-\alpha-1/2+\epsilon})$, $\epsilon > 0$, for functions in $\spaceH$, with $\alpha > 0$, is non-trivial. Simply applying the same proof technique as in~\cite{KKNU2019}, which achieves the improved bound for a similar algorithm but which also averages over the generating vectors in the good set, would not give the improved bound for the new algorithm $\Qranpshift$ here; see~\eqref{eq:badbound}.

The paper is structured as follows.
In Section~\ref{sec:prelim} we will define the error measures and the function space that we will be working with. This leads us to define a set of generating vectors that are good in a certain sense.
In Section~\ref{section:existence_of_vector}, we will present an explicit expression for the \emph{worst case root mean square error} of $\Qranpshift$ and then use it to prove that there exists a pre-determined generating vector that allows to achieve the near optimal \emph{worst case expected error} convergence. We then provide a lower bound on the randomised and the root mean square error for both $\Qranp$ and $\Qranpshift$.

\section{Notation and Preliminaries}\label{sec:prelim}

We define $\N := \{1,2,3,\ldots\}$ and $\Z_n := \{0,1,\ldots,n-1\}$, and we will identify $\Z/n \Z$ with $\Z_n$.
For $\bsh \in \Z^d$ we define $\supp(\bsh):=\{j \in \{1,\ldots,d\} \st h_j \neq 0\}$.
We use $\equiv_p$ as a shorthand for the equivalence relation modulo~$p$.
The notation $\bbmone(\cdot)$ will be used for the indicator function and for $\bsh \in \Z^d$, by $\bbmone(\bsh \equiv_p \bszero)$ we mean that all components of $\bsh$ are multiples of $p$, i.e., $\bbmone(\bsh \equiv_p \bszero) \Leftrightarrow \prod_{j=1}^d \bbmone(h_j \equiv_p 0)$.

\subsection{Function Space and Error Quantities}\label{section:function_space_and_errors}

For $\alpha>0$ and $\bsgamma = \{\gamma_\fraku\}_{\fraku \subset \N}$ with $\gamma_\fraku > 0$, we define, for $d \in \N$,
\begin{equation}\label{eq:H}
    \spaceH
    :=
    \left\{ f \in L_2([0,1]^d) \st \norm{f} < \infty \right\}
    ,
\end{equation}
where
\[
\norm{f}^2 := \sum_{\bsh \in \Z^d} |\hat{f}(\bsh)|^2 \, \ralpha^2(\bsh)
,
\]
and for $\bsh \in \Z^d$ we define
\[
\ralpha(\bsh)
:=
\gamma_{\supp(\bsh)}^{-1} \prod_{j \in \supp(\bsh)} |h_j|^\alpha
\quad
\text{and}
\quad
\hat{f}(\bsh)
:=
\int_{[0,1]^d} f(\bsx) \, \rme^{-2 \pi \rmi \, \bsh \cdot \bsx} \, \mathrm{d}\bsx
.
\]
We consider an infinite sequence of weights, $\bsgamma = \{\gamma_\fraku\}_{\fraku \subset \N}$, to study tractability with respect to $d$, see~\eqref{eq:tractability_condition}.
It is important to note that, in the case where $\alpha > 1/2$, these functions are continuous on the torus $[0, 1)^d$ and can be expressed in terms of an absolutely converging Fourier series (and so we can speak about ``periodic'' functions where $f$ and its derivatives $f^{(\bstau)}$ for $0 \le \tau_j < \alpha - 1/2$ have absolutely converging Fourier series, see also \cite{NS2023}).

The weights $\bsgamma_\fraku$ are used to model the importance of the different subsets of dimensions in the function space, where larger weights denote larger importance.
Weighted spaces were first introduced in \cite{SW1998} and \cite{SW2001}, initially with \emph{product weights} $\gamma_\fraku = \prod_{j\in\fraku} \gamma_j$ for some positive sequence $\{\gamma_j\}_{j\ge1}$ so as to be able to state dimension independent error bounds under certain summability conditions, see \cite{NW2010}.
They will allow us to state our final error bound independent of the dimension.

For a randomised algorithm $A_{d,n}^\textrm{ran}$ we define the \emph{worst case expected error}, also referred to as the \emph{randomised error}, as
\begin{align}\label{eq:eran}
\eran{d}(A_{d,n}^\textrm{ran})
=
\sup_{\substack{f \in \spaceH \\ \norm{f} \leq 1}} \bbmE_\omega\big[|A_{d,n}^\textrm{ran}(f, \omega)-I_d(f)|\big]
.
\end{align}
Our proof will actually provide a bound on the \emph{worst case root mean square error},
\begin{align}\label{eq:erms}
\erms{d}(A_{d,n}^\textrm{ran})
=
\sup_{\substack{f \in \spaceH \\ \norm{f} \leq 1}} \sqrt{\bbmE_\omega\big[|A_{d,n}^\textrm{ran}(f,\omega)-I_d(f)|^2\big]}
,
\end{align}
which can be seen to satisfy
\begin{align*}
\eran{d}(A_{d,n}^\textrm{ran})
\leq
\erms{d}(A_{d,n}^\textrm{ran})
\end{align*}
due to Jensen's inequality.

\subsection{Good Sets}\label{section:good_sets}

For our existence proof we will limit the pre-determined generating vector to be equivalent to certain ``good generating vectors'' when considered modulo the primes in $\Pn$ separately.
For this we define the following sets of good generating vectors.
For $\lambda \in (0,\alpha)$ and prime $p$ define
\begin{align}\label{eq:G}
\Gset{p}
=
G_{d,\alpha,\bsgamma,\lambda}^{(p)}
&:=
\Bigg\{\bsz \in \Z_p^d \st \sum_{\substack{\bsh \in \Zdnull \\ \bsh \cdot \bsz \equiv_{p} 0}} r_{\alpha,\bsgamma}^{-1/\lambda}(\bsh) \leq \frac{4}{p} \muquant \Bigg\}
,
\end{align}
where
\begin{equation}\label{eq:mu_quantity}
    \muquant
    :=
    \sum_{\bsh \in \Zdnull}  \ralpha^{-1/\lambda}(\bsh)
    \qquad < \infty \quad \text{for } \lambda \in (0, \alpha)
    .
\end{equation}
Such ``good sets'', in different forms, were also used in \cite{KKNU2019,DGS2022,KNW2023}.
In \cite{DGS2022,KNW2023} the good sets were build w.r.t.\ the (deterministic) worst case error.
This is an efficiently computable criterion (in the case of a lattice rule with e.g.\ product weights), but it limits $\alpha > 1/2$.
In \cite{KKNU2019} the good set was formulated in terms of the Zaremba index which can be seen as the reciprocal of the $\ell_\infty$ worst case error, allowing for any $\alpha > 0$, but resulting in a hard to calculate criterion, see, e.g., \cite{N2014}.
With our definition~\eqref{eq:G}, we see that the value of $\alpha$ can be less than $1/2$, but we are required to fix a value of $\lambda \in (0,\alpha)$ such that~\eqref{eq:mu_quantity} is finite.
Furthermore, the set~\eqref{eq:G} depends on $\lambda$ in contrast to the good sets in \cite{KKNU2019,DGS2022,KNW2023}.

From the sets $\Gset{p}$, for $p \in \Pn$, we will now build a set $\Gnset$.
Let \[N := \prod_{p \in \Pn} p\] and define $\Gnset \subset \Z_N^d$ as the set of generating vectors which, modulo each of the primes in $\Pn$, are congruent to those in the good sets for each prime, i.e.,
\[
\Gnset
:=
\left\{\bsz \in \Z_N^d \st \big( \forall p \in \Pn \big) \:  \big( \exists ! \, \bsz^{(p)} \in \Gset{p} \big)  \: \text{with} \: \bsz \equiv_p \bsz^{(p)}\right\}
.
\]
By the Chinese remainder theorem, we have the following bijection
$$
\Gnset
\cong
\Gset{p_1} \oplus \cdots \oplus \Gset{p_L}
.
$$

We have the following two lemmas which pertain to the good sets. The first shows that these sets are sufficiently large.
\begin{lemma}\label{lemma:good_sets_are_large}
    For $\lambda \in (0,\alpha)$ and prime $p$, the set $\Gset{p}$ satisfies the condition
    \[
    \left|\Gset{p}\right|
    \geq
    \ceil{\frac{p^d}{2}}
    \geq 1
    .
    \]
\end{lemma}
\begin{proof}
    For $\bsh \in \Z^d$ and prime $p$ we will use
    \begin{equation} \label{eq:average_of_indicator}
        \frac{1}{p^d} \sum_{\bsz \in \Z_p^d} \bbmone(\bsh \cdot \bsz \equiv_p 0)
        =
        \frac{p-1}{p}\bbmone(\bsh \equiv_p \bszero)+\frac{1}{p}
        ,
    \end{equation}
    which comes from the fact that, if there is an $h_j \not\equiv_p 0$, then, fixing all other components of $\bsz$, there is exactly one choice of $z_j$ which results in~$\bsh \cdot \bsz \equiv_p 0$ (since $p$ is prime).
    Now take the average
    \begin{align*}
        &\frac{1}{p^d}\sum_{\bsz \in \Z_p^d} \sum_{\substack{\bsh \in \Zdnull \\ \bsh \cdot \bsz \equiv_{p} 0}} \ralpha^{-1/\lambda}(\bsh) \\
        &\hspace{5mm}= \sum_{\bsh \in \Zdnull} \Bigg(\frac{1}{p^d}\sum_{\bsz \in \Z_p^d} \bbmone(\bsh \cdot \bsz \equiv_{p} 0)\Bigg) \, \ralpha^{-1/\lambda}(\bsh)
        \\
        &\hspace{5mm}= \frac{p-1}{p} \sum_{\bsh \in \Zdnull} \bbmone(\bsh \equiv_{p} \bszero) \, \ralpha^{-1/\lambda}(\bsh) +\frac{1}{p}\sum_{\bsh \in \Zdnull} \ralpha^{-1/\lambda}(\bsh)
        \\
        &\hspace{5mm}\leq \sum_{\bsh \in \Zdnull}   \left(p^{|\supp(\bsh)|}\right)^{-\alpha/\lambda}\ralpha^{-1/\lambda}(\bsh) +\frac{1}{p} \, \muquant
        \\
        &\hspace{5mm}\leq \frac{2}{p} \, \muquant
        .
    \end{align*}
    For the first inequality, $\bsh$ is substituted with $p \bsh$ in the first sum, as all components of $\bsh$ are divisible by~$p$. This allows us to extract the factor of $p$ from $\ralpha^{-1/\lambda}(p\bsh)$.

    By Markov's inequality and the above average over all $\bsz \in \Z_p^d$, we can see that, for $\bsz$ chosen uniformly from~$\Z_p^d$,
    \[
    \bbmP_{\bsz \sim \mathcal{U}(\Z_p^d)}\Bigg(\sum_{\substack{\bsh \in \Zdnull \\ \bsh \cdot \bsz \equiv_{p} 0}} \ralpha^{-1/\lambda}(\bsh) \leq \frac{4}{p} \, \muquant\Bigg)
    \geq
    \frac{1}{2}
    .
    \]
    This implies that, there are at least $\ceil{p^d/2}$ such vectors which satisfy
    \[
    \sum_{\substack{\bsh \in \Zdnull \\ \bsh \cdot \bsz \equiv_{p} 0}} \ralpha^{-1/\lambda}(\bsh)
    \leq
    \frac{4}{p} \, \muquant
    ,
    \]
    which proves the lemma.
\end{proof}

The next lemma allows us to see that certain $\ralpha(\bsh)$ terms will be bounded below.

\begin{lemma}\label{lemma:lower_bound_on_r}
    For $\lambda \in (0,\alpha)$, prime $p \in \Pn$ and $\bsz \in \Gset{p}$, if $\bsh \in \Z_d \setminus \{\bszero\}$ satisfies $\bsh \cdot \bsz \equiv_p 0$, we have
    \begin{align}\label{eq:rbound}
    \ralpha(\bsh)
    >
    B_{n,\lambda}
    ,
	\end{align}
    where
    \begin{align}\label{eq:B_n}
        B_{n,\lambda}
        =
        B_{d,\alpha,\bsgamma,n,\lambda}
        :=
        n^{\lambda} \left[8 \, \muquant \right]^{-\lambda}
        .
    \end{align}
\end{lemma}
\begin{proof}
    By the definition of $\Gset{p}$, and using $p > n/2$ for $p \in \Pn$, we have
    \[
    \sum_{\substack{\bsh \in \Zdnull \\ \bsh \cdot \bsz \equiv_p 0}} \ralpha^{-1/\lambda}(\bsh)
    \leq
    \frac{4}{p} \, \muquant
    <
    \frac{8}{n} \, \muquant
    .
    \]
    This also implies that whenever $\bsh \in \Z_d \setminus \{\bszero\}$ satisfies $\bsh \cdot \bsz \equiv_p 0$ for any $p\in \Pn$,
    \begin{equation*}
    \ralpha(\bsh)
    >
    n^{\lambda} \, (8 \, \muquant)^{-\lambda}
    ,
    \end{equation*}
    thus proving the lemma.
\end{proof}

\section{Existence of a Vector with Near Optimal Error}
\label{section:existence_of_vector}

In this section we show existence of a vector which will produce the near optimal randomised error for $\eran{d}(\Qranpshift)$ by proving a bound on $\erms{d}(\Qranpshift)$, cf.~\eqref{eq:eran} and~\eqref{eq:erms}. We will start with a lemma which reduces the \emph{worst case root mean square error}, see~\eqref{eq:erms}, to a more manageable form.

\begin{lemma}\label{lemma:rms_alt_form}
    For $\alpha > 0$ we have
    \begin{equation}\label{eq:ermsranpshift}
        \erms{d}(\Qranpshift) = \sup_{\bsh \in \Zdnull} \frac{\sqrt{\omegafixed(\bsh)}}{\ralpha(\bsh)}
    \end{equation}
    where
    \[
    \omegafixed(\bsh) := \frac{1}{|\Pn|} \sum_{p \in \Pn} \bbmone(\bsh \cdot \bsz \equiv_p 0)
    .
    \]
\end{lemma}

\begin{proof}
    We first note that, with $\bsDelta \sim \mathcal{U}([0,1)^d$),
    \begin{align}
        \nonumber
        \bbmE_\bsDelta\left[|Q_{d,p,\bsz,\bsDelta}(f)-I_d(f)|^2\right]
        &=
        \bbmE_\bsDelta \Bigg[\bigg|\sum_{\substack{\bsh \in \Zdnull \\ \bsh \cdot \bsz \equiv_p 0}} \hat{f}(\bsh) \, \rme^{2 \pi \rmi \, \bsh \cdot \bsDelta}\bigg|^2\Bigg] \\
        \nonumber
        &=
        \bbmE_\bsDelta \Bigg[\sum_{\substack{\bsh \in \Zdnull \\ \bsh \cdot \bsz \equiv_p 0}} \sum_{\substack{\bsh' \in \Zdnull \\ \bsh' \cdot \bsz \equiv_p 0}} \hat{f}(\bsh) \overline{\hat{f}(\bsh')} \, \rme^{2 \pi \rmi (\bsh-\bsh') \cdot \bsDelta}\Bigg] \\
        \label{eq:mse}
        &=
        \sum_{\substack{\bsh \in \Zdnull \\ \bsh \cdot \bsz \equiv_p 0}} \left|\hat{f}(\bsh)\right|^2
        ,
    \end{align}
    and by Hölder's inequality
    \begin{align*}
        &\frac{1}{|\Pn|} \sum_{p \in \Pn} \sum_{\substack{\bsh \in \Zdnull \\ \bsh \cdot \bsz \equiv_p 0}} \left|\hat{f}(\bsh)\right|^2
        =
        \sum_{\bsh \in \Zdnull} \omegafixed(\bsh) \, \left|\hat{f}(\bsh)\right|^2
        \\
        &\hspace{10mm}\leq
        \Bigg(\sup_{\bsh \in \Zdnull} \omegafixed(\bsh) \, \ralpha^{-2}(\bsh)\Bigg) \,\Bigg(\sum_{\bsh \in \Z^d} \ralpha^2(\bsh) \left|\hat{f}(\bsh)\right|^2\Bigg)
        \\
        &\hspace{10mm}=
        \Bigg(\sup_{\bsh \in \Zdnull} \omegafixed(\bsh) \, \ralpha^{-2}(\bsh)\Bigg) \, \norm{f}^2
        .
    \end{align*}
    This implies that
    \begin{align*}
    \left[\erms{d}(\Qranpshift)\right]^2
    &=
    \sup_{\substack{f \in \spaceH \\ \norm{f} \leq 1}} \frac{1}{|\Pn|} \sum_{p \in \Pn} \sum_{\substack{\bsh \in \Zdnull \\ \bsh \cdot \bsz \equiv_p 0}} \left|\hat{f}(\bsh)\right|^2
    \\
    &\leq
    \sup_{\bsh \in \Zdnull} \omegafixed(\bsh) \, \ralpha^{-2}(\bsh)
    .
    \end{align*}
    To show equality, we define a function
    \[
    \xi(\bsx)
    :=
    \frac{\rme^{2 \pi \rmi \, \bsh^* \cdot \bsx}}{\ralpha(\bsh^*)}
    ,
    \]
    where $\bsh^* \in \Zdnull$ is such that
    \[
    \omegafixed(\bsh^*) \, \ralpha^{-2}(\bsh^*)
    =
    \sup_{\bsh \in \Zdnull} \omegafixed(\bsh) \, \ralpha^{-2}(\bsh)
    =
    \max_{\substack{\bsh \in \Zdnull}}
    \omegafixed(\bsh) \, \ralpha^{-2}(\bsh)
    .
    \]
    This function is in the space $\spaceH$ and has integral $0$, norm $1$, and a single positive Fourier coefficient.
    Moreover, its root mean square error when using $\Qranpshift$,
    \begin{align*}
      &
      \sqrt{\bbmE_\omega[|Q_{d,p(\omega_1),\,\bsz,\bsDelta(\omega_2)}(\xi) - I_d(\xi)|^2]}
      \\
      &\qquad=
      \sqrt{\frac{1}{|\Pn|} \sum_{p \in \Pn} \int_{[0,1)^d} \left| \frac1p \sum_{k=0}^{p-1} r^{-1}_{\alpha,\bsgamma}(\bsh^*) \, \rme^{2 \pi \rmi \, \bsh^* \cdot \bsz k / p} \, \rme^{2 \pi \rmi \, \bsh^* \cdot \bsDelta} \right|^2 \, \mathrm{d}\bsDelta }
      \\
      &\qquad=
      \sqrt{\frac{1}{|\Pn|} \sum_{p \in \Pn} \int_{[0,1)^d} r^{-2}_{\alpha,\bsgamma}(\bsh^*) \, \bbmone(\bsh^* \cdot \bsz \equiv_p 0) \, \left| \rme^{2 \pi \rmi \, \bsh^* \cdot \bsDelta} \right|^2 \, \mathrm{d}\bsDelta }
      \\
      &\qquad=
      \sqrt{\frac{1}{|\Pn|} \sum_{p \in \Pn} r^{-2}_{\alpha,\bsgamma}(\bsh^*) \, \bbmone(\bsh^* \cdot \bsz \equiv_p 0) }
      =
      \sqrt{\omegafixed(\bsh^*) \, \ralpha^{-2}(\bsh^*)}
      ,
    \end{align*}
    is equal to the upper bound for $\erms{d}(\Qranpshift)$, implying that the upper bound can be attained and hence the worst case root mean square error is given by~\eqref{eq:ermsranpshift}.
\end{proof}

At this point the most obvious route for the proof would be to replace the root mean square error with an upper bound by noting that $\ralpha(\bsh) > B_{n,\lambda}$, see~\eqref{eq:rbound}, as is also done in \cite[Section~3.2]{KKNU2019} for an algorithm which also averages over the set of good generating vectors.
However, in our case, where we are interested in a pre-determined generating vector, the proof technique from \cite{KKNU2019} would provide the following bound
\begin{align}\label{eq:badbound}
\erms{d}(\Qranpshift)
&\leq
\frac{1}{B_{n,\lambda}} \sup_{\bsh \in \Zdnull} \omegafixed(\bsh)
\nonumber
\\
&=
\frac{1}{n^\lambda} \left( 8\muquant \right)^\lambda
,
\quad \forall \lambda \in (0,\alpha)
,
\end{align}
which is clearly not optimal in our case since we hope to achieve a convergence of approximately $O(n^{-\alpha-1/2})$. Instead we will use a slight relaxation of the supremum over $\bsh \in \Zdnull$ such that we can take the average over all possibilities of $\bsh \in \Zdnull$.
This requires us however to make sure the sum, given below in~\eqref{eq:relaxedbound}, converges.

We first need some ingredients to be able to prove our main theorem, which we will present next.
We make use of the prime number theorem \cite{lvP1896,lvP1899}, which allows us to bound the size of the set $\Pn$.
For $n \ge 2$, there exist two positive constants $C_1$ and $C_2$ such that
\begin{equation}\label{eq:PNT}
    C_1 \, \frac{n}{\ln(n)}
    \leq
    |\Pn|
    \leq
    C_2 \, \frac{n}{\ln(n)}
    .
\end{equation}
These constants can be taken closer and closer to $0.5$ when $n$ is assumed to be larger. For $n \geq 20$, we can take $C_1$ and $C_2$ that satisfy $0.4 < C_1 < 0.5 < C_2 < 0.6$.
We further use the notation $\stirling{r}{k}$ to denote Stirling numbers of the second kind, which give the number of ways of partitioning $r$ distinct elements into $k$ unlabelled non-empty partitions. We use $\boldsymbol{T}_r(\cdot)$ to denote the Touchard polynomial of order $r$ and $\boldsymbol{B}_r$ to denote the $r$th Bell number. The Touchard polynomials are given by
\[
\boldsymbol{T}_r(x)
=
\sum_{k=0}^r \stirling{r}{k} \, x^k
\]
and the Bell numbers count the partitions of a set. Crucially, we will use the fact that
\begin{equation} \label{eq:bound_on_touchard}
    \boldsymbol{T}_r(x)
    \leq
    \sum_{k=0}^r
    \stirling{r}{k}
    =
    \boldsymbol{B}_r
    \quad \text{where} \quad 0 \leq x \leq 1
    ,
\end{equation}
and the inequality
\begin{equation} \label{eq:bound_on_bell}
\boldsymbol{B}_r < \left(\frac{C_3 \, r}{\ln(r+1)}\right)^r
,
\end{equation}
where $C_3 = 0.792$, see \cite{BT2010}.

\begin{theorem} \label{theorem:existence_result}
    For $d \in \N$, $\alpha > 0$, positive weights $\{\gamma_\fraku\}_{\fraku \subset \N}$ and for each $n \in \N$ satisfying $n \geq \rme^{6 C_2}>34$, with $N = \prod_{p \in \Pn} p$, there exists a vector $\bsz \in \Z_N^d$ which achieves the bound
    \[
    \erms{d}(\Qranpshift)
    \leq
    \frac{1}{n^{\lambda+(r-1)/(2 r)}} \left( \frac{C_3 \, r \, \ln(n)}{C_1 \, \ln(r+1)} \right)^{1/2} (4 \, \muquant )^{\lambda}
    \]
    for all $0 < \lambda < \alpha$ and $r \in \N$ with $r \ge 1/(2\lambda)$. The constants $C_1$ and $C_2$ can be taken to satisfy $0.4 < C_1 < 0.5 < C_2 < 0.6$, cf.\ \eqref{eq:PNT}, and we then fix $C_3 = 0.792$.
    Moreover, if
    \begin{align}
      \sum_{\substack{\emptyset \ne \fraku \subset \N \\ |\fraku| < \infty}} \gamma_\fraku^{-1/\lambda} \, (2 \zeta(\alpha/\lambda))^{|\fraku|}
      <
      \infty
      \label{eq:tractability_condition}
    \end{align}
    then the upper bound can be stated with an absolute constant independent of the dimension.
\end{theorem}

\begin{proof}
    Using the result of Lemma~\ref{lemma:rms_alt_form}, we can write
    \begin{align}
        \left[\erms{d}(\Qranpshift)\right]^2
        &=
        \sup_{\bsh \in \Zdnull} \omegafixed(\bsh) \, \ralpha^{-2}(\bsh) \nonumber \\
        &\leq
        \Bigg(\sum_{\bsh \in \Zdnull} \omegafixed^{r}(\bsh) \, \ralpha^{-2 r}(\bsh)\Bigg)^{1/r}
        , \qquad \forall r > 1/(2\alpha)
        . \label{eq:relaxedbound}
    \end{align}
    We will specify that $r$ is an integer so that we can use the following equality which uses $L = |\Pn|$,
    \begin{align}\label{eq:omega-r-expanded}
    \omegafixed^r(\bsh)
    &=
    \frac{1}{L^r} \left(\sum_{p \in \Pn} \bbmone(\bsh \cdot \bsz \equiv_p 0) \right)^r
    \nonumber
    \\
    &=
    \frac{1}{L^r}\sum_{\substack{\emptyset \neq \fraku \subseteq \{1,\ldots,L\} \\ |\fraku| \leq r}} \labstirling{r}{|\fraku|} \prod_{i \in \fraku} \bbmone(\bsh \cdot \bsz \equiv_{p_i} 0)
    ,
    \end{align}
    where, for $r, m \in \N$,
    \begin{equation*}
    \labstirling{r}{m}
    :=
    \sum_{\substack{\ell_1,\ldots,\ell_m \ge 1 \\ \ell_1 + \cdots + \ell_m = r}} \frac{r!}{\prod_{i=1}^m \ell_i!}
    =
    m! \, \stirling{r}{m}
    \end{equation*}
    is the number of ways of partitioning $r$ distinct elements into $m$ labelled non-empty partitions. Each $\emptyset \neq \fraku \subseteq \{1,\ldots,L\}$ corresponds to a subset of $\Pn$ and so, with $a(r,|\fraku|)$, we count the number of times the product $\prod_{i \in \fraku} \bbmone(\bsh \cdot \bsz \equiv_{p_i} 0)$ appears when we expand the $r$th power of the sum noting that $\bbmone(\bsh \cdot \bsz \equiv_p 0)^2 = \bbmone(\bsh \cdot \bsz \equiv_p 0)$. We note also that both forms of $a(r,m)$ can be used to see~\eqref{eq:omega-r-expanded} and their equivalence is given by \cite[Proposition~5.5]{MT2016}.

    We begin by bounding the error in the case that the generating vector $\bsz$ lies in~$\Gnset$:
    \begin{align*}
        \left[\erms{d}(\Qranpshift)\right]^{2 r}
        &\le
        \sum_{\substack{\bsh \in \Zdnull \\ \ralpha(\bsh) > B_{n,\lambda}}} \omegafixed^r(\bsh) \, \ralpha^{-2 r}(\bsh)
        \\
        &\leq
        \sum_{\substack{\bsh \in \Zdnull \\ \ralpha(\bsh) > B_{n,\lambda}}} \omegafixed^r(\bsh) \, \ralpha^{-2 r}(\bsh) \left(\frac{\ralpha(\bsh)}{B_{n,\lambda}}\right)^{2 r-1/\lambda}
        \\
        &\leq
        \frac{1}{B_{n,\lambda}^{2 r-1/\lambda}} \sum_{\bsh \in \Zdnull} \omegafixed^r(\bsh) \, \ralpha^{-1/\lambda}(\bsh)
        \qquad\quad \forall r \geq 1/(2\lambda)
        .
    \end{align*}
    This uses the fact that if $\bsz \in \Gnset$ then $\omegafixed(\bsh) \neq 0$ implies $\ralpha(\bsh) > B_{n,\lambda}$, see~\eqref{eq:rbound}. It also requires the strengthened condition on $r$ that it is greater than or equal to $1/(2\lambda)$.

    Taking the average of the quantity $\big[\erms{d}(\Qranpshift)\big]^{2 r}$ produces the following bound
    \begin{align*}
        &\frac{1}{|\Gnset|} \sum_{\bsz \in \Gnset} \left[\erms{d}(\Qranpshift)\right]^{2 r}
        \\
        &\hspace{4mm}\leq
        \frac{1}{|\Gnset|} \sum_{\bsz \in \Gnset} \frac{1}{B_{n,\lambda}^{2 r-1/\lambda}} \sum_{\bsh \in \Zdnull} \omegafixed^r(\bsh) \, \ralpha^{-1/\lambda}(\bsh)
        \\
        &\hspace{4mm}=
        \frac{1}{L^r}\sum_{\substack{\emptyset \neq \fraku \subseteq \{1,\ldots,L\} \\ |\fraku| \leq r}} \frac{\labstirling{r}{|\fraku|}}{B_{n,\lambda}^{2 r-1/\lambda}} \sum_{\bsh \in \Zdnull} \prod_{i \in \fraku} \frac{1}{|\Gset{p_i}|} \sum_{\bsz^{(p_i)} \in \Gset{p_i}} \hspace{-10pt} \bbmone(\bsh \cdot \bsz \equiv_{p_i} 0) \, \ralpha^{-1/\lambda}(\bsh)
        \\
        &\hspace{4mm}\leq
        \frac{1}{L^r}\sum_{\substack{\emptyset \neq \fraku \subseteq \{1,\ldots,L\} \\ |\fraku| \leq r}} \frac{\labstirling{r}{|\fraku|}}{B_{n,\lambda}^{2 r-1/\lambda}} \sum_{\bsh \in \Zdnull} \prod_{i \in \fraku} \left( \bbmone(\bsh \equiv_{p_i} \bszero) + \frac{2}{p_i} \right) \ralpha^{-1/\lambda}(\bsh)
        \\
        &\hspace{4mm}\leq
        \frac{1}{L^r}\sum_{\substack{\emptyset \neq \fraku \subseteq \{1,\ldots,L\} \\ |\fraku| \leq r}} \frac{\labstirling{r}{|\fraku|}}{B_{n,\lambda}^{2 r-1/\lambda}} \sum_{\bsh \in \Zdnull} \sum_{\frakv \subseteq \fraku} \left(\prod_{i \in \frakv} \bbmone(\bsh \equiv_{p_i} \bszero) \right) \left(\frac{4}{n}\right)^{|\fraku\setminus\frakv|} \ralpha^{-1/\lambda}(\bsh)
        .
    \end{align*}
    Here, we have used the inequality
    \begin{align*}
    \frac{1}{|\Gset{p}|} \sum_{\bsz^{(p)} \in \Gset{p}} \bbmone(\bsh \cdot \bsz^{(p)} \equiv_p 0)
    &\leq
    \bbmone(\bsh \equiv_p \bszero)+\frac{1}{|\Gset{p}|}\bbmone(\bsh \not\equiv_p \bszero)
    \\
    &\leq
    \bbmone(\bsh \equiv_p \bszero)+\frac{2}{p}
    ,
    \end{align*}
    which is an altered version of~\eqref{eq:average_of_indicator} and makes use of Lemma~\ref{lemma:good_sets_are_large}.
    We now utilise the fact that all components of $\bsh$ are divisible by all primes $p_i$ with $i \in \frakv$ to see the following
    \begin{align*}\label{eqn:star}
    \sum_{\bsh \in \Zdnull} \left(\prod_{i \in \frakv} \bbmone(\bsh \equiv_{p_i} \bszero) \right) \ralpha^{-1/\lambda}(\bsh)
    &=
    \sum_{\bsh \in \Zdnull} \bbmone\Big(\bsh \equiv \bszero \bmod{\prod_{i \in \frakv} p_i}\Big) \, \ralpha^{-1/\lambda}(\bsh)
    \\
    &=
    \sum_{\bsh \in \Zdnull} \ralpha^{-1/\lambda}\bigg(\big(\prod_{i \in \frakv}p_i\big)\bsh\bigg)
    \\
    &\leq
    \left(\frac{2}{n}\right)^{|\frakv|} \muquant
    .
    \tag{*}
    \end{align*}
    Also, we have
    \begin{align*}\label{eqn:starstar}
        \sum_{\substack{ m = 1 \\ m \leq r}}^{L} m! \, \binom{L}{m} \stirling{r}{m} \left(\frac{6}{n}\right)^{m}
        &\leq
        \sum_{m = 1}^{r} m! \, \binom{L}{m}\stirling{r}{m} \left(\frac{6}{n}\right)^{m}
        \\
        &\leq
        \sum_{m = 1}^{r} \stirling{r}{m} \left(\frac{6 L}{n}\right)^{m}
        \\
        &\leq
        \boldsymbol{T}_r\left(\frac{6 C_2}{\ln(n)}\right)
        \\
        &\leq
        \boldsymbol{B}_r
        \\
        &\leq
        \left(\frac{C_3 \, r}{\ln(r+1)}\right)^r
        ,
        \tag{**}
    \end{align*}
    where we use \eqref{eq:PNT}, both \eqref{eq:bound_on_touchard} and \eqref{eq:bound_on_bell}, and the assumption that $n \geq \rme^{6 C_2}$ (i.e.\ $\ln(n) \geq 6 C_2$) for the penultimate inequality.
    These inequalities will be used in the following step.
    Continuing with the bound on the average of $\left[\erms{d}(\Qranpshift)\right]^{2 r}$, we have
    \begin{align*}
        &\frac{1}{|\Gnset|} \sum_{\bsz \in \Gnset} \left[\erms{d}(\Qranpshift)\right]^{2 r}
        \\
        &\hspace{5mm}\stackrel{(\refeq{eqn:star})}{\leq}
        \frac{1}{L^r}\sum_{\substack{\emptyset \neq \fraku \subseteq \{1,\ldots,L\} \\ |\fraku| \leq r}} \frac{\labstirling{r}{|\fraku|}}{B_{n,\lambda}^{2 r-1/\lambda}} \sum_{\frakv \subseteq \fraku} \left(\frac{2}{n} \right)^{|\frakv|} \left(\frac{4}{n}\right)^{|\fraku\setminus\frakv|} \muquant
        \\
        &\hspace{5mm}\leq
        \frac{1}{L^r \, B_{n,\lambda}^{2 r-1/\lambda}}\sum_{\substack{ m = 1 \\ m \leq r}}^{L} m! \, \binom{L}{m} \stirling{r}{m} \left(\frac{6}{n}\right)^{m} \muquant
        \\
        &\hspace{5mm}\stackrel{(\refeq{eqn:starstar})}{\leq}
        \frac{1}{L^r \, B_{n,\lambda}^{2 r-1/\lambda}} \left(\frac{C_3 \, r}{\ln(r+1)}\right)^r \muquant
        \\
        &\hspace{5mm}\leq
        \frac{1}{n^r \, B_{n,\lambda}^{2 r-1/\lambda}} \left(\frac{C_3 \, r \ln(n)}{C_1 \ln(r+1)}\right)^r \muquant
        \\
        &\hspace{5mm}=
        \frac{1}{8 n^{2 r\lambda+r-1}} \left(\frac{C_3 \, r \ln(n)}{C_1 \ln(r+1)}\right)^r \left(8 \muquant\right)^{2 r \lambda}
        .
    \end{align*}
    Since $\Gnset \subseteq \Z_N^d$ for all $\lambda \in (0,\alpha)$, we can see that the minimiser $\bsz \in \Z_N^d$ of $\erms{d}(\Qranpshift)$ satisfies, for all $\lambda \in (0,\alpha)$ and $r \in \N$ with $r \ge 1/(2\lambda)$,
    \begin{align*}
        \erms{d}(\Qranpshift)
        &\leq
        \left(\frac{1}{8 n^{2 r \lambda+r-1}} \left(\frac{C_3 \, r \ln(n)}{C_1 \ln(r+1)}\right)^r \left(8 \sum_{\bsh \in \Zdnull} \ralpha^{-1/\lambda}(\bsh)\right)^{2 r \lambda}\right)^{1/(2 r)} \\
        &\leq
        \frac{1}{n^{\lambda+(r-1)/(2 r)}} \left(\frac{C_3 \, r \ln(n)}{C_1 \ln(r+1)}\right)^{1/2} \left(8 \sum_{\bsh \in \Zdnull} \ralpha^{-1/\lambda}(\bsh)\right)^{\lambda}
        .
    \end{align*}

    Additionally, when~\eqref{eq:tractability_condition} holds, we have
    \[
      \muquant
      =
      \sum_{\bsh \in \Zdnull} \ralpha^{-1/\lambda}(\bsh)
      <
      \sum_{\substack{\emptyset \ne \fraku \subset \N \\ |\fraku| < \infty}} \gamma_\fraku^{-1/\lambda} \, (2 \zeta(\alpha/\lambda))^{|\fraku|}
      <
      \infty
      ,
    \]
    for all $d \in \N$, and so the bound can be stated with an absolute constant independent of the dimension.
\end{proof}

This result shows the error bound
\[
\erms{d}(\Qranpshift)
\leq
\frac{1}{n^{\lambda+1/2-1/(2r)}} \left( \frac{C_3 \, r \, \ln(n)}{C_1 \, \ln(r+1)} \right)^{1/2} (4 \, \muquant )^{\lambda}
,
\]
which holds for any $r \in \N$ with $r \ge 1/(2\lambda)$.
So, for any fixed $\epsilon > 0$, we can now pick $\lambda = \alpha - \epsilon/2$ and fix $r \in \N$ such that both $1/(2r) \le \epsilon/2$ and $r \ge 1/(2\lambda) = 1/(2\alpha-\epsilon)$ are satisfied.
Since we obviously want $\epsilon \ll \alpha$ we just have to be concerned about the second condition which asks to take $r \ge 1/\epsilon$.
Together with the fact that for any $m \ge 0$ and $\epsilon > 0$ we have $[\ln(n)]^m \in o(n^\epsilon) \subset O(n^\epsilon)$ we obtain, for any $\epsilon > 0$,
\[
\erms{d}(\Qranpshift)
\in
O(n^{-\alpha-1/2+\epsilon})
.
\]
A similar statement holds for $\eran{d}(\Qranpshift)$ since
\[\eran{d}(\Qranpshift) \le \erms{d}(\Qranpshift)\]
by Jensen's inequality.

A lower bound for $\Qranp$ and $\Qranpshift$ can be obtained similarly as in~\cite{KKNU2019}.

\begin{theorem}
    Let $d \in \N$, $\alpha \ge 0$ and positive weights $\{\bsgamma_\fraku\}_{\fraku \subset \N}$. For $n \ge 2$ and any $\bsz \in \Z^d$, the randomised errors of the algorithms $\Qranp$ and $\Qranpshift$ are bounded from below by
    \[
    \erms{d}(\Qranp) \geq \eran{d}(\Qranp) \geq \frac{\gamma_{\{1\}} \sqrt{\ln(n)} }{\sqrt{C_2} \, n^{\alpha + 1/2}}
    \]
    and
    \[
    \erms{d}(\Qranpshift)
    \geq
    \frac{\gamma_{\{1\}} \sqrt{\ln(n)}}{\sqrt{C_2} \, n^{\alpha+1/2}}
    .
    \]
    If $z_1$ is coprime to all $p \in \Pn$, then
    \[
        \eran{d}(\Qranpshift)
        \geq
        \frac{\gamma_{\{1\}} \sqrt{\ln(n)}}{\sqrt{C_2} \, n^{\alpha+1/2}}
        .
    \]

\end{theorem}
\begin{proof}
    For any $n \ge 2$, define the function
    \begin{align*}
        f_n(\bsx)
        &:=
        \sum_{h_1 \in \Pn} \left( \ralpha(h_1, 0, \ldots, 0) \, \sqrt{|\Pn|} \right)^{-1} \, \rme^{2 \pi \rmi \, h_1 x_1}
        .
    \end{align*}
    This function only has positive Fourier coefficients, $I_d(f) = 0$ and $\norm{f} = 1$ for all $\alpha \ge 0$ and thus its error can be used as a lower bound.
    For any $p \in \Pn$,
    \begin{align*}
        | Q_{d,p,\bsz}(f_n) - I_d(f_n) |
        &=
        \sum_{\substack{h_1 \in \Pn \\ h_1 z_1 \equiv_p 0}} \left( \ralpha(h_1, 0, \ldots, 0) \, \sqrt{|\Pn|} \right)^{-1}
        \\
        &\ge
        \left( \ralpha(p, 0, \ldots, 0) \, \sqrt{|\Pn|} \right)^{-1}
        ,
    \end{align*}
    for which the last line is equality if $z_1$ and $p$ are relatively prime since then only the term for $h_1 = p$ survives.
    Also
    \begin{align*}
        | Q_{d,p,\bsz,\bsDelta}(f_n) - I_d(f_n) |
        &=
        \Bigl| \sum_{\substack{h_1 \in \Pn \\ h_1 z_1 \equiv_p 0}} \left( \ralpha(h_1, 0, \ldots, 0) \, \sqrt{|\Pn|} \right)^{-1} \, \rme^{2 \pi \rmi \, h_1 \Delta_1} \Bigr|
        \\
        &\stackrel{\mathclap{(*)}}{=}
        \left( \ralpha(p, 0, \ldots, 0) \, \sqrt{|\Pn|} \right)^{-1}
        ,
    \end{align*}
    where the last equality, marked with the star, only holds if $z_1$ and $p$ are relatively prime.
    Lastly, for arbitrary~$\bsz$, from~\eqref{eq:mse},
    \begin{align*}
        \bbmE_\bsDelta\bigl[ | Q_{d,p,\bsz,\bsDelta}(f_n) - I_d(f_n) |^2 \bigr]
        &=
        \sum_{\substack{h_1 \in \Pn \\ h_1 z_1 \equiv_p 0}} \left( \ralpha(h_1, 0, \ldots, 0) \, \sqrt{|\Pn|} \right)^{-2}
        \\
        &\ge
        \left( \ralpha(p, 0, \ldots, 0) \, \sqrt{|\Pn|} \right)^{-2}
        ,
    \end{align*}
    where again the last line is equality if $z_1$ and $p$ are relatively prime.
    The above expressions become independent of $p$ by using $p \le n$ for $p \in \Pn$ and the upper bound on $|\Pn|$ given by~\eqref{eq:PNT}. These are then independent of the random elements in the algorithms and hence we arrive at the stated lower bounds.
\end{proof}

\bigskip

\textbf{Acknowledgements.}\quad \mbox{}
We would like to acknowledge the support from the Research Foundation Flanders (FWO G091920N).
We thank the referees and Frances Y.\ Kuo for discussions on this paper.

\bibliographystyle{abbrv}
\bibliography{Bibliography.bib}

\begin{thebibliography}{10}

\bibitem{B1961}
N.~S. Bakhvalov.
\newblock An estimate of the mean remainder term in quadrature formulae.
\newblock {\em Zhurnal Vychislitel'noi Matematiki i Matematicheskoi Fiziki},
  1:64--77, 1961.

\bibitem{BT2010}
D.~Berend and T.~Tassa.
\newblock Improved bounds on {Bell} numbers and on moments of sums of random
  variables.
\newblock {\em Probability and Mathematical Statistics}, 30(2):185--205, 2010.

\bibitem{lvP1896}
C.~J. de~La~Vall{\'e}e~Poussin.
\newblock Recherches analytiques sur la théorie des nombres premiers.
  {P}remière partie. {L}a fonction $\zeta(s)$ de {Riemann} et les nombres
  premiers en général, suivi d’un appendice sur des réflexions applicables
  à une formule donnée par {Riemann}.
\newblock {\em Ann. Soc. Scient. Bruxelles, deuxi{\`e}me partie}, 20:183--256,
  1896.

\bibitem{lvP1899}
C.~J. de~La~Vall{\'e}e~Poussin.
\newblock Sur la fonction $\zeta(s)$ de {Riemann} et le nombre des nombres
  premiers inf{\'e}rieurs {\`a} une limite donn{\'e}e.
\newblock {\em M{\'e}moires couronn{\'e}s et autres M{\'e}moires in-8
  publi{\'e}s par l'Acad{\'e}mie royale de Belgique}, 49:74, 1899.

\bibitem{D2007}
J.~Dick.
\newblock Higher order scrambled digital nets achieve the optimal rate of the
  root mean square error for smooth integrands.
\newblock {\em The Annals of Statistics}, 39(3):1372--1398, 2011.

\bibitem{DGS2022}
J.~Dick, T.~Goda, and K.~Suzuki.
\newblock Component-by-component construction of randomized rank-1 lattice
  rules achieving almost the optimal randomized error rate.
\newblock {\em Mathematics of Computation}, 91:2771--2801, 2022.

\bibitem{DKP2022}
J.~Dick, P.~Kritzer, and F.~Pillichshammer.
\newblock {\em Lattice Rules: Numerical Integration, Approximation, and
  Discrepancy}.
\newblock Springer International Publishing, 2022.

\bibitem{F1976}
K.~K. Frolov.
\newblock Upper error bounds for quadrature formulas on function classes.
\newblock {\em Dokl. Akad. Nauk SSSR}, 231(4):818--821, 1976.

\bibitem{GlE2022}
T.~Goda and P.~L'Ecuyer.
\newblock Construction-free median quasi-{M}onte {C}arlo rules for function
  spaces with unspecified smoothness and general weights.
\newblock {\em SIAM Journal on Scientific Computing}, 44(4):A2765--A2788, 2022.

\bibitem{KN2017}
D.~Krieg and E.~Novak.
\newblock A universal algorithm for multivariate integration.
\newblock {\em Foundations of Computational Mathematics}, 17:895--916, 2017.

\bibitem{KKNU2019}
P.~Kritzer, F.~Y. Kuo, D.~Nuyens, and M.~Ullrich.
\newblock Lattice rules with random $n$ achieve nearly the optimal
  $\mathcal{O}(n^{-\alpha-1/2})$ error independently of the dimension.
\newblock {\em Journal of Approximation Theory}, 240:96--113, 2019.

\bibitem{KNW2023}
F.~Y. Kuo, D.~Nuyens, and L.~Wilkes.
\newblock Random-prime–fixed-vector randomised lattice-based algorithm for
  high-dimensional integration.
\newblock {\em Journal of Complexity}, 2023.
\newblock To appear.

\bibitem{MT2016}
C.~Mariconda and A.~Tonolo.
\newblock {\em Discrete Calculus: Methods for Counting}.
\newblock Springer International Publishing, 2016.

\bibitem{N1992}
H.~Niederreiter.
\newblock {\em Random Number Generation and Quasi-{M}onte {C}arlo Methods}.
\newblock Society for Industrial and Applied Mathematics, 1992.

\bibitem{NW2008}
E.~Novak and H.~Wo{\'z}niakowski.
\newblock {\em Tractability of Multivariate Problems: Volume I: Linear
  information}.
\newblock EMS Tracts in Mathematics. European Mathematical Society, 2008.

\bibitem{NW2010}
E.~Novak and H.~Wo{\'z}niakowski.
\newblock {\em Tractability of Multivariate Problems: Volume II: Standard
  Information for Functionals}.
\newblock EMS Tracts in Mathematics. European Mathematical Society, 2010.

\bibitem{N2014}
D.~Nuyens.
\newblock {\em The construction of good lattice rules and polynomial lattice
  rules}, pages 223--256.
\newblock De Gruyter, Berlin, Boston, 2014.

\bibitem{NS2023}
D.~Nuyens and Y.~Suzuki.
\newblock Scaled lattice rules for integration on $\mathbb{R}^d$ achieving
  higher-order convergence with error analysis in terms of orthogonal
  projections onto periodic spaces.
\newblock {\em Mathematics of Computation}, 92:307--347, 2023.

\bibitem{O1995}
A.~B. Owen.
\newblock Randomly permuted $(t,m,s)$-nets and $(t, s)$-sequences.
\newblock In H.~Niederreiter and P.~J.-S. Shiue, editors, {\em Monte Carlo and
  Quasi-Monte Carlo Methods in Scientific Computing}, pages 299--317, New York,
  NY, 1995. Springer New York.

\bibitem{O1997}
A.~B. Owen.
\newblock Monte {C}arlo variance of scrambled net quadrature.
\newblock {\em SIAM Journal on Numerical Analysis}, 34(5):1884--1910, 1997.

\bibitem{SJ1994}
I.~H. Sloan and S.~Joe.
\newblock {\em Lattice Methods for Multiple Integration}.
\newblock Oxford University Press, 1994.

\bibitem{SKJ2002}
I.~H. Sloan, F.~Y. Kuo, and S.~Joe.
\newblock On the step-by-step construction of quasi-{M}onte {C}arlo integration
  rules that achieve strong tractability error bounds in weighted {Sobolev}
  spaces.
\newblock {\em Mathematics of Computation}, 71:1609--1640, 2002.

\bibitem{SW1998}
I.~H. Sloan and H.~Wo{\'z}niakowski.
\newblock When are quasi-{M}onte {C}arlo algorithms efficient for high
  dimensional integrals?
\newblock {\em Journal of Complexity}, 14(1):1--33, 1998.

\bibitem{SW2001}
I.~H. Sloan and H.~Woźniakowski.
\newblock Tractability of multivariate integration for weighted {Korobov}
  classes.
\newblock {\em Journal of Complexity}, 17(4):697--721, 2001.

\bibitem{U2017}
M.~Ullrich.
\newblock A {M}onte {C}arlo method for integration of multivariate smooth
  functions.
\newblock {\em SIAM Journal on Numerical Analysis}, 55:1188--1200, 2017.

\end{thebibliography}

\end{document}